\documentclass[11pt,reqno]{amsart}

\usepackage{amsthm,amsmath,amssymb,booktabs,xcolor,graphicx}
\usepackage{bm,bbm}
\usepackage{listings}
\usepackage{xcolor}
\usepackage{appendix}
\usepackage{enumerate}
\usepackage[shortlabels]{enumitem}

\usepackage[margin=1in]{geometry}
\usepackage[utf8]{inputenc}
\usepackage[T1]{fontenc}

\usepackage[textsize=small,backgroundcolor=orange!20]{todonotes}
\usepackage[hidelinks]{hyperref}
\usepackage{url}
\usepackage{etoolbox}
\usepackage{appendix}
\usepackage[noabbrev,capitalize]{cleveref}
\crefname{equation}{}{} 
\crefname{subsection}{Subsection}{Subsection} 
\AtBeginEnvironment{appendices}{\crefalias{section}{appendix}} 

\usepackage[color,final]{showkeys} 

\colorlet{refkey}{orange!20}
\colorlet{labelkey}{blue!30}

\numberwithin{equation}{section}
\newtheorem{theorem}{Theorem}[section]
\newtheorem{proposition}[theorem]{Proposition}
\newtheorem{lemma}[theorem]{Lemma}

\crefname{claim}{Claim}{Claims}

\newtheorem*{question*}{Question}

\theoremstyle{definition}
\newtheorem{definition}[theorem]{Definition}

\newtheorem*{definition*}{Definition}

\theoremstyle{remark}
\newtheorem*{remark}{Remark}

\newcommand{\snorm}[1]{\lVert#1\rVert}

\newcommand{\mb}{\mathbb}
\newcommand{\mbm}{\mathbbm}
\newcommand{\mbf}{\mathbf}

\newcommand{\ol}{\overline}
\newcommand{\on}{\operatorname}

\allowdisplaybreaks

\title{Diagonal Ramsey via effective quasirandomness}
\author[Sah]{Ashwin Sah}

\thanks{}

\address{Department of Mathematics, Massachusetts Institute of Technology, Cambridge, MA 02139, USA}
\email{asah@mit.edu}

\begin{document}
\begin{abstract}
We improve the upper bound for diagonal Ramsey numbers to
\[R(k+1,k+1)\le\exp(-c(\log k)^2)\binom{2k}{k}\]
for $k\ge 3$. To do so, we build on a quasirandomness and induction framework for Ramsey numbers introduced by Thomason and extended by Conlon, demonstrating optimal ``effective quasirandomness'' results about convergence of graphs. This optimality represents a natural barrier to improvement.
\end{abstract}

\maketitle

\section{Introduction}\label{sec:introduction}
The Ramsey number $R(k,\ell)$, introduced by Ramsey \cite{Ram29} in relation to logic, is the smallest positive integer $n$ such that every graph on $n$ vertices contains a subgraph isomorphic to $K_k$ (the complete graph on $k$ vertices) or $\ol{K}_\ell$ (the empty graph on $\ell$ vertices), i.e., contains a clique of size $k$ or an independent set of size $\ell$. Erd\H{o}s and Szekeres \cite{ES35} gave a classic upper bound
\begin{equation}\label{eq:erdos-szekeres}
R(k+1,\ell+1)\le\binom{k+\ell}{k}.
\end{equation}
This stood for a long time, until R\"odl (unpublished) in the $1980$'s showed
\[R(k+1,\ell+1)\le\frac{\binom{k+\ell}{k}}{c\log^c(k+\ell)}\]
for some $c > 0$. The weaker bound $6\binom{k+\ell}{k}/\log\log(k+\ell)$ appears in a survey on Ramsey theory by Graham and R\"odl \cite{GR87}.

Thomason \cite{Tho88} showed that for some $A > 0$
\[R(k+1,\ell+1)\le k^{-\ell/(2k)+A/\sqrt{\log k}}\binom{k+\ell}{k}\]
when $k\ge\ell$, which is an improvement of a polynomial factor in $k$ over the Erd\H{o}--Szekeres bound for $k,\ell$ of the same order. This stood until Conlon \cite{Con09} demonstrated
\[R(k+1,k+1)\le k^{-c\log k/\log\log k}\binom{2k}{k}\]
and a similar upper bound for $R(k+1,\ell+1)$ for $k,\ell$ roughly the same order.

We improve this result. In doing so we develop a variety of tools for handling effective quasirandomness estimates, allowing us to extract optimal estimates. As we will discuss in \cref{sub:introduction-quasirandomness}, this in particular represents the natural limit of a quasirandomness framework for bounding Ramsey numbers initiated by Thomason \cite{Tho88}.
\begin{theorem}\label{thm:main}
There is an absolute constant $c > 0$ such that for $k\ge 3$,
\[R(k+1,k+1)\le e^{-c(\log k)^2}\binom{2k}{k}.\]
\end{theorem}
As with \cite{Tho88,Con09}, we prove a result in the general regime where $k,\ell$ are roughly the same size.
\begin{theorem}\label{thm:main-2}
For each $\varepsilon\in(0,1/2)$ there is $c_\varepsilon > 0$ such that
\[R(k+1,\ell+1)\le e^{-c_\varepsilon(\log k)^2}\binom{k+\ell}{k}\]
whenever $\ell/k\in[\varepsilon,1]$ and $\ell\ge c_\varepsilon^{-1}$.
\end{theorem}
Note that \cref{thm:main-2} immediately implies \cref{thm:main} (the condition $k\ge 3$ is put merely to ensure that $R(k+1,k+1) < \binom{2k}{k}$). Henceforth we will restrict our attention to this result.

\subsection{The Ramsey problem}\label{sub:introduction-ramsey}
The standard Erd\H{o}--Szekeres \cite{ES35} proof goes as follows. It is enough to show that $R(k+1,\ell+1)\le R(k,\ell+1)+R(k+1,\ell)$. If we have a graph on $R(k,\ell+1)+R(k+1,\ell)$ vertices, then any vertex $v$ has either at least $R(k,\ell+1)$ neighbors or at least $R(k+1,\ell)$ non-neighbors by the pigeonhole principle. Restricting to the first case and applying the definition of the Ramsey numbers, we see that we can find either a clique of size $k$ attached to $v$ (hence a clique of size $k+1$) or an independent set of size $\ell+1$, and similar in the other case. This finishes the proof.

It is evident that the crux in this argument is the simple fact that if a graph avoids $K_{k+1}$ and $\ol{K}_{\ell+1}$ (we call this a \emph{Ramsey graph}), then every vertex is adjacent to at most $R(k,\ell+1)-1$ vertices and non-adjacent to at most $R(k+1,\ell)-1$ vertices.

Thomason's \cite{Tho88} approach, on which Conlon's \cite{Con09} is based, considers the following stronger property of Ramsey graphs: every clique of size $r$ extends to at most $R(k+1-r,\ell+1)$ cliques of size $r+1$, and every independent set of size $r$ extends to at most $R(k+1,\ell+1-r)$ independent sets.

We make this idea more explicit here. Let $\alpha(k,\ell)$ be a slowly decaying function, and suppose that we know $R(a+1,b+1)\le \alpha(a,b)\binom{a+b}{a}$ for all $a+b < k + \ell$. Then, defining $\alpha^\ast(a,b) = \lfloor \alpha(a,b)\binom{a+b}{a}\rfloor/\binom{a+b}{a}$, the Erd\H{o}--Szekeres argument (equivalently $r=1$ of the observation above) generalizes to show all the degrees in a Ramsey graph $G$ on $n = f^\ast(k,\ell)\binom{k+\ell}{k}$ vertices are in
\[\left[\left(1-\frac{\alpha(k,\ell-1)}{\alpha^\ast(k,\ell)}\cdot\frac{\ell}{k+\ell}\right)n,\left(\frac{\alpha(k-1,\ell)}{\alpha^\ast(k,\ell)}\cdot\frac{k}{k+\ell}\right)n\right).\]
If $\alpha$ decays slowly enough, this implies that the degrees of $G$ are close to $pn$, where $p = k/(k+\ell)$. Similarly, one can give an upper bound on the total number of $K_3$ and $\ol{K}_3$ that is close to ``expected'' for a quasirandom graph of density $p$. On the other hand, this sum is controlled purely by the degree sequence (Goodman's formula). If $\alpha$ decays slowly enough, one derives a contradiction.

Conlon builds on this idea by computing the number of $K_r$ and $\ol{K}_r$ for larger $r$. Instead of exactly counting this sum via the degree sequence, he showed that $G$ must be quasirandom in an appropriate sense, and then showed that $K_r$ and $\ol{K}_r$ have counts near ``expected'' (controlling for major sources of deviation such as the edge count and triangle count). This gives a contradiction for $\alpha$ decaying quicker than in Thomason's argument.

We extend this framework by developing tools for effective quasirandomness, including optimal control of $H$-densities of graphs that are suitably regular. This allows us to extend the range of $r$ to which we can control the $K_r$- and $\ol{K}_r$-densities out to the optimal scale, which leads to an improved bound. See the discussion in \cref{sub:introduction-quasirandomness} for more discussion of effective quasirandomness.

\subsection{Other Ramsey results}\label{sub:introduction-other}
We do not focus on other well-studied natural variants of the Ramsey problem, including hypergraph Ramsey, multicolor Ramsey, Ramsey for subgraphs other than cliques, ordered Ramsey, explicit (non-random) constructions of Ramsey graphs, and other variants. We also remain concerned with the regime where $k,\ell$ are roughly the same size, although the regime where $\ell$ is constant has seen significant study. See \cite{CFS15} for a comprehensive survey of Ramsey theory, and \cite{Rad94} for a dynamic survey of Ramsey theory for small numbers.

We quickly remark on the lower bound for diagonal Ramsey numbers. It is a classic application of the probabilistic method by Erd\H{o}s \cite{E47} that
\[R(k,k)\ge(1+o(1))\frac{1}{e\sqrt{2}}k2^{k/2}.\]
This was improved by a factor of two by Spencer \cite{Spe77} to $R(k,k)\ge (1+o(1))k2^{(k+1)/2}/e$ using the Lov\'asz local lemma. This is where the lower bound remains.

\subsection{Quasirandomness and regularity}\label{sub:introduction-quasirandomness}
As mentioned, our treatment of the Ramsey problem involves the development of many effective quasirandomness tools. The notion of quasirandomness in graph theory dates back at least to Thomason \cite{Tho87} and Chung, Graham, and Wilson \cite{CGW89}. The correct notion of quasirandomness is that a graph $G$ has close to $p^4n^4$ ordered cycles, where $n = |V(G)|$. Here $p$ is assumed to be fixed and $n$ growing. This definition allows one to show that $G$ has approximately $p^{e(H)}n^{v(H)}$ ordered copies of the subgraph $H$ if $H$ has fixed size, i.e., gives a counting lemma.

We will study the setting in which we wish to show a graph $G$ behaves like $G(n,p)$, as above. However, we would be remiss if we did not discuss the related concept of graph regularity. Szemer\'edi's regularity lemma \cite{Sze78} demonstrates that every graph is $\varepsilon$-``close'' to a graph composed of a finite number of pieces which behave in a quasirandom manner between most pairs of pieces. Given a regular partition, one can also count subgraphs. However, the explicit dependence on $\varepsilon$ is quite bad, with potentially $\on{tow}(\varepsilon^{-c})$ pieces needed \cite{Gow97} (here $\on{tow}$ means taking an exponential tower of $2$'s of the specified length).

Much attention has been given to generalizing this result to hypergraphs (e.g. \cite{RNSSK05,Gow07}) and sparse graphs (e.g. \cite{GS05,BMS15,CGSS14,Sch16,CFZ14,CFZ15}) in connection with extremal graph theory and additive combinatorics. However, less attention has been given to giving sharp effective bounds in the setting where $n = |V(G)|$ does not escape to infinity and the size of the counted subgraph is allowed to grow, depending on the quality of the quasirandomness. The work of Thomason \cite{Tho88} and Conlon \cite{Con09} on the Ramsey problem can be interpreted in this light, and we devote a significant portion of this paper to explicitly working out such ``effective quasirandomness'' results before applying them to bound Ramsey numbers. (There is also work of Lov\'asz \cite{Lov11} on local Sidorenko inequalities which studies similar quantities to this line of work.)

Consider the following setup. We have a graph $G$ on $n$ vertices which in some sense has codegrees of pairs of vertices close to $p^2n$. Can we control the subgraph density of $H$, where $H$ has a growing number of vertices? The degrees of closeness, $\mu_{p,G}$ and $\nu_{p,G}$, are defined in \cref{sub:graph-effective}. It turns out that if $\mu$ and $\nu$ are decaying exponentially in $v(H)$ and $n$ is growing roughly square exponentially in $v(H)$, then we can give such effective bounds, which we state in \cref{thm:effective-distance}. Furthermore, the dependence between $\nu$ and $v(H)$ is optimal, as we demonstrate in \cref{sub:optimality}. This optimality implies that to improve the bound on Ramsey numbers, new ideas will be required.

To prove this theorem, we decompose the $H$-density of $G$ into a combination of $H'$-densities for a signed graphon $W_G-p$ which captures the distance between $G$ and $G(n,p)$. These $H'$-densities are in turn controlled via upper bounds by $K_{2,a}$-densities for various $a$. One can interpret Conlon's result as bounding these $H'$-densities by a ``local'' contribution, whereas our results piece together ``global'' contributions in order to obtain optimal bounds.

The idea of bounding graph densities above by $K_{a,b}$-densities has appeared in the study of independent sets and graph homomorphisms, including work of the author, Sawhney, Stoner, and Zhao \cite{SSSZ20} on reverse Sidorenko inequalities. However, we require inequalities that work for signed functions, which is a major departure from the methods in the graph homomorphism literature, and we require inequalities valid for $H$-densities when $H$ has triangles, which is often a significant hurdle in that line of work (see the survey of Zhao \cite{Zha17} for an overview of that area).

\subsection{Notation}\label{sub:notation}
We will write $[n] = \{1,\ldots,n\}$. The notations $o,O,\omega,\Omega$ have their usual asymptotic meanings, and $f = \Theta(g)$ means $f = O(g)$ and $g = O(f)$. Subscripts imply dependence of the implicit constants on the subscript. We use $K_n$ and $\ol{K}_n$ to denote the complete and empty graphs on $n$ vertices, respectively, and $P_n$ to mean the path with $n$ edges and $n+1$ vertices.

\subsection{Outline}\label{sub:outline}
In \cref{sec:graphons}, we give some background and conventions regarding graphons. In \cref{sec:effective-quasirandomness}, we establish key effective quasirandomness results, building up to \cref{thm:effective-distance}. We additionally prove that the relation between the quasirandomness and the size of subgraphs counted is optimal. In \cref{sec:ramsey}, we use these effective quasirandomness results in conjunction with the framework developed by Thomason \cite{Tho88} and Conlon \cite{Con09} to deduce \cref{thm:main-2}.

\subsection*{Acknowledgements}
We thank David Conlon and Yufei Zhao for helpful comments on the manuscript.

\section{Graphons}\label{sec:graphons}
We recall some basic notions from the theory of graph limits. A \emph{graphon} is a symmetric measurable function $W\colon [0,1]^2\to[0,1]$. These arise naturally as the limits of dense graphs under an appropriate topology (convergence of subgraph densities, or equivalently convergence in cut norm). Any graph $G$ with $n$ vertices gives a graphon in a natural way. Label its vertices by $[n]$, then partition $[0,1]^2$ into $n^2$ squares of equal dimensions, labeled $(i,j)$ for $i,j\in[n]$, and assign the value $1$ to $W$ on a block $(i,j)$ if and only if $i,j$ are adjacent in $G$.

For convenience, we will define a slightly different ``graphon'' associated to $G$. Let $\Omega_{V(G)}$ be the measure space on $V(G)$ assigning probability $1/|V(G)|$ to each element. Then associated to $G$ is the symmetric measurable function $W_G\colon\Omega_{V(G)}^2\to[0,1]$ given by
\[W_G(i,j) = \mbm{1}_{(i,j)\in E(G)}.\]
In general, we will find it more convenient to work with general (bounded) symmetric measurable functions until we restrict our study to graphs. However, we will not remark on unimportant measure-theoretic concerns (such as an inequality holding almost everywhere versus everywhere).

Let $\Omega$ be a measure space of total measure $1$. We define, for a graph $H$ and bounded symmetric measurable function $W\colon\Omega^2\to\mb{C}$, the $H$\emph{-density}
\[t_H(W) = \int_{\mbf{x}}\prod_{v_1v_2\in E(H)}W(x_{v_1},x_{v_2})\,d\mbf{x}.\]
Here the variable is $\mbf{x} = (x_v)_{v\in V(H)}$, and $d\mbf{x}$ is the product measure on $\Omega^{V(H)}$. These conventions will often go unstated in the future, and we will write $\int_{\mbf{x}}$ as $\mb{E}_{\mbf{x}}$, using the now standard expectation notation prevalent in extremal combinatorics.

We also use the following notation for ``codegrees''. If $\mbf{x} = (x_1,\ldots,x_r)$ then
\[W_{\mbf{x}} = \mb{E}_y\bigg[\prod_{i=1}^rW(x_i,y)\bigg].\]

\section{Effective Quasirandomness}\label{sec:effective-quasirandomness}
We develop effective quasirandomness estimates that will be needed later, and then demonstrate their optimality.

\subsection{Bounding densities of signed graphons}\label{sub:graphon-effective}
We bound densities $t_H(W)$ of bounded symmetric measurable functions $W\colon\Omega^2\to\mb{C}$ in terms of bipartite graph densities, with the aim of reducing control of convergence to control over codegrees of pairs of vertices.

We first establish some inequalities regarding $t_{K_{a,b}}$. It is worth noting that
\[t_{K_{a,b}}(W) = \mb{E}_{x_1,\ldots,x_a}W_{x_1,\ldots,x_a}^b,\]
as we will use this and similar expansions for $t_H$ when $H$ is bipartite repeatedly. This expression and the symmetry of $a,b$ immediately show that $t_{K_{a,b}}(W)\ge 0$ if $ab$ is even.
\begin{lemma}\label{lem:kab}
If $W\colon\Omega^2\to\mb{C}$ satisfies $\snorm{W}_\infty\le 1$ and $a,b,c$ are positive integers with $a\ge c$ and $c$ even, then
\[|t_{K_{a,b}}(W)|\le |t_{K_{c,b}}(W)|\]
\end{lemma}
\begin{proof}
Create variables $\mbf{x} = (x_i)_{i\in[a]}$ and $\mbf{y} = (y_j)_{j\in[b]}$. We note $\snorm{W}_\infty\le 1$ implies $|W_{\mbf{y}}|\le 1$. Hence
\[|t_{K_{a,b}}(W)| = \bigg|\mb{E}_{\mbf{x},\mbf{y}}\bigg[\prod_{i\in[a]}\prod_{j\in[b]}W(x_i,y_j)\bigg]\bigg| = |\mb{E}_{\mbf{y}}W_{\mbf{y}}^a|\le\mb{E}_{\mbf{y}}|W_{\mbf{y}}|^a\le\mb{E}_{\mbf{y}}W_{\mbf{y}}^c = t_{K_{c,b}}(W).\qedhere\]
\end{proof}
We next establish a weak ``local'' bound for $t_H(W)$ in terms of these statistics. This estimate essentially appears in \cite{Con09}.
\begin{proposition}\label{prop:local-bound}
If $W\colon\Omega^2\to\mb{C}$ satisfies $\snorm{W}_\infty\le 1$ and $H$ is a graph containing a vertex of degree $d$, then
\[|t_H(W)|\le|t_{K_{2,d}}(W)|^{1/2}.\]
\end{proposition}
\begin{proof}
Let $v = |V(H)|$ and $V(H) = [v]$. Consider variables $\mbf{x}=(x_i)_{i\in[v]}$. Without loss of generality, suppose the vertex $1$ has degree $d$ and has neighborhood $N(d) = \{2,\ldots,d+1\}$. Let $\mbf{x}_{-1} = (x_i)_{2\le i\le v}$ and $\mbf{y} = (x_i)_{2\le i\le d+1}$. If $E'$ is the set of edges of $H$ not including the vertex $1$, then we have
\begin{align*}
|t_H(W)| &= \bigg|\mb{E}_{\mbf{x}}\bigg[\prod_{ij\in E(H)}W(x_i,x_j)\bigg]\bigg| = \bigg|\mb{E}_{\mbf{x}_{-1}}\bigg[\prod_{ij\in E'}W(x_i,x_j)\mb{E}_{x_1}\bigg[\prod_{i=2}^{d+1}W(x_1,x_i)\bigg]\bigg]\bigg|\\
&\le\mb{E}_{\mbf{x}_{-1}}|W_{\mbf{y}}| = \mb{E}_{\mbf{y}}|W_{\mbf{y}}|\le|\mb{E}_{\mbf{y}}W_{\mbf{y}}^2|^{1/2} = |t_{K_{2,d}}(W)|^{1/2}.\qedhere
\end{align*}
\end{proof}
We now establish a ``global'' bound for $t_H(W)$ in terms of these statistics, first in the bipartite case.
\begin{proposition}\label{prop:bipartite-global-bound}
If $W\colon\Omega^2\to\mb{C}$ satisfies $\snorm{W}_\infty\le 1$ and $H$ is a bipartite graph with bipartition $V(H) = A\sqcup B$ such that $|B| = h$ and $B$ has no vertices of degree at most $1$, then
\[|t_H(W)|\le|t_{K_{2,2\lceil h/2\rceil}}(W)|^{h/(2\lceil h/2\rceil)}.\]
\end{proposition}
\begin{proof}
For $v\in V(H)$ let $N(v)$ be its neighborhood. Create variables $\mbf{x}=(x_a)_{a\in A}$ and $\mbf{x}'=(x_a')_{a\in A}$. Write $\mbf{x}_T = (x_a)_{a\in T}$ for a set $T\subseteq A$. Let $u = \lceil h/2\rceil$. We have
\[|t_H(W)|=\bigg|\mb{E}_{\mbf{x}}\prod_{b\in B}W_{\mbf{x}_{N(b)}}\bigg|\le\prod_{b\in B}\bigg|\mb{E}_{\mbf{x}}W_{\mbf{x}_{N(b)}}^{2u}\bigg|^{1/(2u)} = \prod_{b\in B}|t_{K_{|N(b)|,2u}}(W)|^{1/(2u)}\]
by definition; H\"older's inequality; and definition, respectively. Applying \cref{lem:kab}, since $|N(b)|\ge 2$ for $b\in B$ we deduce
\[|t_H(W)|\le\prod_{b\in B}|t_{K_{2,2u}}(W)|^{1/(2u)} = |t_{K_{2,2u}}(W)|^{h/(2u)}.\qedhere\]
\end{proof}
Finally, we establish a ``global'' bound for $t_H(W)$ in the general case.
\begin{proposition}\label{prop:general-global-bound}
If $W\colon\Omega^2\to\mb{C}$ satisfies $\snorm{W}_\infty\le 1$ and $H$ is a graph with $h$ vertices and no isolated vertices, then
\[|t_H(W)|\le|t_{K_{2,2\lceil h/2\rceil}}(W)|^{1/4}.\]
\end{proposition}
\begin{proof}
Consider a partition $V(H) = A\sqcup B$ such that every $v\in A$ has a neighbor in $B$ and every $v\in B$ has a neighbor in $A$. This is easily done, for example, by $2$-coloring any spanning forest of $H$. Let $E(A)$ and $E(B)$ be the set of edges of $H$ internal to $A$ and $B$, respectively, and let $E(A,B)$ be the set of cross-edges. We create variables $\mbf{x}=(x_a)_{a\in A}$, $\mbf{x}' = (x_a')_{a\in A}$ and $\mbf{y} = (y_b)_{b\in B}$, $\mbf{y}' = (y_b')_{b\in B}$. Then
\begin{align*}
|t_H(W)|^4 &= \bigg|\mb{E}_{\mbf{x},\mbf{y}}\bigg[\prod_{a_1a_2\in E(A)}W(x_{a_1},x_{a_2})\prod_{ab\in E(A,B)}W(x_a,y_b)\prod_{b_1b_2\in E(B)}W(y_{b_1},y_{b_2})\bigg]\bigg|^4\\
&\le\bigg|\mb{E}_{\mbf{x}}\bigg[\bigg|\mb{E}_{\mbf{y}}\bigg[\prod_{ab\in E(A,B)}W(x_a,y_b)\prod_{b_1b_2\in E(B)}W(y_{b_1},y_{b_2})\bigg]\bigg|\bigg]\bigg|^4\\
&\le\bigg|\mb{E}_{\mbf{x}}\bigg[\bigg|\mb{E}_{\mbf{y}}\bigg[\prod_{ab\in E(A,B)}W(x_a,y_b)\prod_{b_1b_2\in E(B)}W(y_{b_1},y_{b_2})\bigg]\bigg|^2\bigg]\bigg|^2\\
&=\bigg|\mb{E}_{\mbf{x},\mbf{y},\mbf{y}'}\bigg[\prod_{ab\in E(A,B)}W(x_a,y_b)W(x_a,y_b')\prod_{b_1b_2\in E(B)}W(y_{b_1},y_{b_2})W(y_{b_1}',y_{b_2}')\bigg]\bigg|^2\\
&\le\bigg|\mb{E}_{\mbf{y},\mbf{y}'}\bigg[\bigg|\mb{E}_{\mbf{x}}\bigg[\prod_{ab\in E(A,B)}W(x_a,y_b)W(x_a,y_b')\bigg]\bigg|\bigg]\bigg|^2\\
&\le\mb{E}_{\mbf{y},\mbf{y}'}\bigg[\bigg|\mb{E}_{\mbf{x}}\bigg[\prod_{ab\in E(A,B)}W(x_a,y_b)W(x_a,y_b')\bigg]\bigg|^2\bigg]\\
&=\mb{E}_{\mbf{x},\mbf{x}',\mbf{y},\mbf{y}'}\bigg[\prod_{ab\in E(A,B)}W(x_a,y_b)W(x_a,y_b')W(x_a',y_b)W(x_a',y_b')\bigg],
\end{align*}
using the definition of $t_H$; rearrangement and $\snorm{W}_\infty\le 1$; Cauchy--Schwarz; expansion; rearrangement and $\snorm{W}_\infty\le 1$; Cauchy--Schwarz; and expansion, respectively.

Define bipartite graph $H'$ with $V(H') = (A\cup B)\times\{0,1\}$ such that: first, $(a,c)$ neighbors $(b,d)$ when $a\in A$, $b\in B$, and $ab\in E(A,B)$, and second, these are all its edges. We have shown so far that
\[|t_H(W)|^4\le t_{H'}(W).\]
Now note $V(H') = A'\sqcup B'$, where $A' = A\times\{0,1\}$ and $B' = B\times\{0,1\}$, and this partition respects the bipartite structure of $H'$. For $v\in V(H')$, let $N(v)$ be the neighborhood of $v$ in $H'$. By the choice of the original bipartition $V(H) = A\sqcup B$ and the definition of $H'$, we see that $|N(v)|\ge 2$ for all $v\in V(H')$.

Create variables $\mbf{z} = (z_a)_{a\in A'}$, and write $\mbf{z}_T = (z_a)_{a\in T}$ for a set $T\subseteq A'$. Then because $|B'| = 2|B|$ is even,
\begin{align*}
|t_{H'}(W)| &= \bigg|\mb{E}_{\mbf{z}}\prod_{b\in B'}W_{\mbf{z}_{N(b)}}\bigg|\le\prod_{b\in B'}\bigg|\mb{E}_{\mbf{z}}W_{\mbf{z}_{N(b)}}^{|B'|}\bigg|^{1/|B'|}=\prod_{b\in B'}|t_{K_{|N(b)|,|B'|}}(W)|^{1/|B'|}\\
&\le\prod_{b\in B'}|t_{K_{2,|B'|}}(W)|^{1/|B'|} = |t_{K_{2,2|B|}}(W)|
\end{align*}
by H\"older's inequality and \cref{lem:kab}, using $|N(b)|\ge 2$ for $b\in B'$. Thus we conclude that
\[|t_H(W)|^4\le|t_{H'}(W)|\le|t_{K_{2,2|B|}}(W)|.\]
We can repeat the same argument with the roles of $A'$ and $B'$ switched, which demonstrates the same with $|B|$ replaced by $|A|$. In particular, we have
\[|t_H(W)|^4\le|t_{H'}(W)|\le|t_{K_{2,2\max(|A|,|B|)}}(W)|\le |t_{K_{2,2\lceil h/2\rceil}}(W)|,\]
the third inequality by applying \cref{lem:kab} again, noting $\max(|A|,|B|)\ge\lceil h/2\rceil$ since we have $|A| + |B| = h$.
\end{proof}

\subsection{Effective convergence of graph densities}\label{sub:graph-effective}
We use the bounds established in \cref{sub:graphon-effective} to demonstrate effective convergence rates of subgraph densities for somewhat large subgraphs. We first establish some notation that will see continued use.
\begin{definition}\label{def:centered-graph}
Given a graph $G$ with $n$ vertices and $p\in(0,1)$, let $f_{p,G}(x,y) = W_G(x,y) - p$, and let
\[\mu_{p,G} = \max_{x\in V(G)}|\mb{E}_yf_{p,G}(x,y)|,\qquad\nu_{p,G} = \max_{x\neq y\in V(G)}\max(0,\mb{E}_zf_{p,G}(x,z)f_{p,G}(z,y)).\]
The expectations can be rewritten $(f_{p,G})_x$ and $(f_{p,G})_{x,y}$, respectively. Note that the latter definition does not take an absolute value, so is one-sided.
\end{definition}
One could prove analogues of the following results given a two-sided guarantee, but this version has slightly more power and will be needed in our study of the Ramsey problem.

We have $\snorm{f_{p,G}}_\infty\le\max(p,1-p)\le 1$. We first give bounds for $K_{2,a}$-densities of $f_{p,G}$ in terms of these statistics. The following lemma is similar to bounds appearing in \cite{Con09}.
\begin{proposition}\label{prop:graph-k2b}
If $G$ is a graph on $n$ vertices and $a\ge 1$ then
\[|t_{K_{2,a}}(f_{p,G})|\le 2\nu_{p,G}^a+2n^{-2/3}.\]
\end{proposition}
\begin{proof}
Let $f = f_{p,G}$. First suppose $a = 2b$ is even, and let $P\subseteq V(G)^2$ be the set of $(x,y)$ with $f_{x,y} = \mb{E}_zf(x,z)f(z,y)\ge 0$. Note that $2$ is even, so
\[0\le t_{K_{2,2b+1}}(f) = \mb{E}_{x,y}f_{x,y}^{2b+1} = \mb{E}_{x,y}[\mbm{1}_{(x,y)\in P}|f_{x,y}|^{2b+1} - \mbm{1}_{(x,y)\notin P}|f_{x,y}|^{2b+1}].\]
Thus
\[\mb{E}_{x,y}|f_{x,y}|^{2b+1}\le 2\mb{E}_{x,y}\mbm{1}_{(x,y)\in P}f_{x,y}^{2b+1}\le 2(\nu_{p,G}^{2b+1}+n^{-1}),\]
the last inequality by definition of $\nu_{p,G}$ and since $|f_{x,x}|\le 1$, using that the event $x = y$ occurs with $1/n$ probability (recall $x,y$ are uniform over $V(G)$).

Thus
\[|t_{K_{2,2b}}(f_{p,G})|\le\mb{E}_{x,y}|f_{x,y}|^{2b}\le|\mb{E}_{x,y}|f_{x,y}|^{2b+1}|^{\frac{2b}{2b+1}}\le(2\nu_{p,G}^{2b+1}+2n^{-1})^{\frac{2b}{2b+1}}\le 2\nu_{p,G}^{2b}+2n^{-2/3}\]
by the triangle inequality; H\"older's inequality; the above; and the well-known inequality $(x+y)^q\le x^q+y^q$ for $q\in(0,1)$ and $x,y\ge 0$.

On the other hand, if $a=2b-1$ is odd, then
\[0\le t_{K_{2,2b-1}}(f_{p,G}) = \mb{E}_{x,y}f_{x,y}^{2b-1}\le\nu_{p,G}^{2b-1}+n^{-1},\]
the first inequality since $2$ is even, and the second by the definition of $\nu_{p,G}$ along with the fact that $x = y$ occurs with $1/n$ probability. We also implicitly used that $x\mapsto x^{2b-1}$ is a monotonic function on $\mb{R}$.
\end{proof}
Now we use the global bounds from \cref{sub:graphon-effective} to effectively bound the distance between $t_H(W_G)$ and $p^{e(H)}$ for graphs $H$.
\begin{theorem}\label{thm:effective-distance}
Let $H$ be a graph on $r$ vertices with $e(H)$ edges and $\Delta_H$ triangles. Let $G$ be a graph on $n$ vertices and let $p\in(0,1)$. Choose some $\nu\in[\nu_{p,G},1]$ so that $r\le\log(\nu^{-1})/(12\log(8/p))$ and $\nu^{-2r}\le n$, and let $f = f_{p,G}$, $\mu = \mu_{p,G}$. Then
\[\bigg|\frac{t_H(W_G)}{p^{e(H)}} - 1 - p^{-1}e(H)t_{K_2}(f) - p^{-3}\Delta_Ht_{K_3}(f)\bigg|\le2^{-2r}\nu^{7/6} + 3\binom{r+1}{4}p^{-2}\mu^2.\]
\end{theorem}
\begin{proof}
Let $H$ have $\Gamma_H\le 3\binom{r}{3}$ unordered paths of length $2$ and $D_H\le 3\binom{r}{4}$ pairs of disjoint edges, which we denote by $K_2+K_2$. Let $C_{H,J} = \#\{H'\subseteq H\colon H'\simeq J\}$, i.e., the number of subgraphs of $H$ isomorphic to $J$. First we note that the number of subgraphs of $H$ with $s$ vertices is upper bounded by $\binom{r}{s}2^{\binom{s}{2}}$.

Choose variables $\mbf{x} = (x_v)_{v\in V(H)}$. Writing $W_G = p + f$ and expanding, we find
\[t_H(W_G) = \mb{E}_{\mbf{x}}\prod_{uv\in E(H)}(p + f(x_u,x_v)) = \sum_Jp^{e(H)-e(J)}C_{H,J}t_J(f),\]
where the sum is over isomorphism classes of graphs $J$ (with $e(J)$ edges) having no isolated vertices. Moving over the terms corresponding to the empty graph, single edge, and triangle, we obtain
\begin{align}
|t_H(W_G) &- p^{e(H)} - p^{e(H)-1}e(H)t_{K_2}(f) - p^{e(H)-3}\Delta_Ht_{K_3}(f)|\label{eq:density-discrepancy}\\
&\le p^{e(H)-2}(\Gamma_H|t_{K_{1,2}}(f)|+D_H|t_{K_2+K_2}(f)|)+{\sideset{}{^\ast}\sum_J} p^{e(H)-e(J)}C_{H,J}|t_J(f)|.\notag
\end{align}
The starred sum is over $J\notin\{K_0,K_2,K_3,K_{1,2},K_2+K_2\}$ with at most $r$ vertices and no isolated vertices. Note that all such $J$ satisfy $|V(J)|\ge 4$.

For $s = |V(J)|\ge 5$, we have by \cref{prop:general-global-bound,prop:graph-k2b} along with the defining property of $\nu$ that
\[|t_J(f)|\le|t_{K_{2,2\lceil s/2\rceil}}(f)|^{1/4}\le(2\nu^{2\lceil s/2\rceil} + 2n^{-2/3})^{1/4}\le2^{1/2}\nu^{\lceil s/2\rceil/2}\le 4\nu\cdot\nu^{s/12},\]
using $n\ge\nu^{-2r}$ and $5\le s\le r$ (we implicitly use $r\ge 5$, but since the term only exists in this case it is fine). For $s = |V(J)| = 4$, either the maximum degree of $J$ is $3$ or $J\in\{K_{2,2},P_3,K_2+K_2\}$ (here $K_2+K_2$ is the disjoint union of two edges). In the first case, \cref{prop:local-bound,prop:graph-k2b} give
\[|t_J(f)|\le|t_{K_{2,3}}(f)|^{1/2}\le(2\nu^3 + 2n^{-2/3})^{1/2}\le 2\nu^{3/2}\le 4\nu\cdot\nu^{s/12}.\]
If $J = K_{2,2}$, the conditions of \cref{prop:bipartite-global-bound} are satisfied and we find
\[|t_J(f)|\le|t_{K_{2,4}}(W)|\le 2\nu^2 + 2n^{-2/3}\le 4\nu^2\le 4\nu\cdot\nu^{s/12}.\]
If $J = P_3$, we find
\begin{align*}
|t_{P_3}(f)| &= |\mb{E}_{x,y}f_{x,y}f_x|\le|\mb{E}_{x,y}f_{x,y}^2|^{1/2}|\mb{E}_{x,y}f_x^2|^{1/2} = |t_{K_{2,2}}(f)|^{1/2}|t_{K_{1,2}}(f)|^{1/2}\\
&\le (2\nu^2+2n^{-2/3})^{1/2}(2\nu+2n^{-2/3})^{1/2}\le 4\nu\cdot\nu^{s/12},
\end{align*}
using Cauchy--Schwarz and similar arguments. Finally, $K_2+K_2$ and $K_{1,2}$ can be bounded via
\[|t_{K_2+K_2}(f)| = |\mb{E}_{x,y}f_xf_y|\le\mu^2,\qquad |t_{K_{1,2}}(f)| = |\mb{E}_{x}f_x^2|\le\mu^2.\]

Overall, we deduce using $e(J)\le\binom{s}{2}$ if $|V(J)| = s$ as well as the upper bound on the number of subgraphs of size $s$ that
\begin{align*}
\frac{\cref{eq:density-discrepancy}}{p^{e(H)}}-3\binom{r}{3}p^{-2}\mu^2-3\binom{r}{4}p^{-2}\mu^2&\le\nu\sum_{s=4}^r4\binom{r}{s}(2/p)^{\binom{s}{2}}\nu^{s/12}\le\nu\sum_{s=4}^r2^{rs/2}(2/p)^{s(s-1)/2}\nu^{s/12}\\
&\le\nu\sum_{s=4}^r2^{-(r+1)s/2}\nu^{s/24}\le p^{e(H)}\nu^{7/6}2^{-2r}
\end{align*}
as the condition on $\nu$ gives $\nu\le(p/8)^{12r}$. We deduce
\[\bigg|\frac{t_H(W_G)}{p^{e(H)}} - 1 - p^{-1}e(H)t_{K_2}(f) - p^{-3}\Delta_Ht_{K_3}(f)\bigg|\le2^{-2r}\nu^{7/6} + 3\binom{r+1}{4}p^{-2}\mu^2.\qedhere\]
\end{proof}
We note that the constants in this result are treated very cavalierly, but even so are still reasonable. This allows control of graph counts of size $r$ so long as our ``codegree control'' $\nu_{p,G}$ is of inverse exponential order in $r$, and our ``degree control'' $\mu_{p,G}$ is of inverse polynomial order or better.

\subsection{Optimality of subgraph size}\label{sub:optimality}
We show that requiring $\nu$ to be inverse exponential size in $r$ is in fact necessary, demonstrating the optimality of \cref{thm:effective-distance}.

Let $p = 1/2$. Choose some $m\ge 1$ and let $W\colon [0,1]^2\to[0,1]$ be defined by
\[W(x,y) = \frac{1 + \mbm{1}_{\lfloor mx\rfloor=\lfloor my\rfloor}}{2}.\]
Now choose $n$ much larger than $m$ and sample a random $W$-random graph $G$. Explicitly, for each $i\in[n]$ we sample $x_i\sim\on{Unif}[0,1]$ independently and then let $V(G) = [n]$, including edge $ij$ independently with probability $W(x_i,x_j)$.

We see that
\[\mu_{1/2,G} = \Theta(m^{-1}),\qquad\nu_{1/2,G} = \Theta(m^{-1})\]
with high probability if $n$ is sufficiently large in terms of $m$, e.g. by multiple applications of Chernoff. Furthermore, by the standard theory of $W$-random graphons, we have
\begin{equation}\label{eq:construction-convergence}
t_{K_r}(W_G)\to t_{K_r}(W) = 2^{-\binom{r}{2}}\sum_J2^{e(J)}t_J(W-1/2)\#\{H'\subseteq H\colon H'\simeq J\},
\end{equation}
as $n\to\infty$, summing over isomorphism classes of graphs $J$ with at most $r$ vertices and no isolated vertices.

But $f_{1/2,G} = W-1/2$ is a block graphon with $m$ square blocks of dimensions $1/m$ along the diagonal. Therefore we have
\[t_J(W-1/2) = 2^{-e(J)}m^{1-v(J)}\]
if $J$ is connected and has $v(J)$ vertices. We therefore see that the contribution to the right hand side of \cref{eq:construction-convergence} from connected graphs $J$ with $r$ vertices is at least $2^{-\binom{r}{2}}\cdot 2^{\binom{r-1}{2}}m^{1-r}$, since there are at least $2^{\binom{r-1}{2}}$ connected subgraphs of $K_r$ with $r$ vertices. Thus
\[2^{\binom{r}{2}}t_{K_r}(W_G) - 1 - 2\binom{r}{2}t_{K_2}(f_{1/2,G}) - 8\binom{r}{3}t_{K_3}(f_{1/2,G})\ge 2^{\binom{r-1}{2}}m^{1-r}\]
for $n$ sufficiently large. For $m\le 2^{r/4}$, this error is growing (and in particular outstrips the ``lower order'' terms corresponding to $K_2,K_3$, etc.), which is in direct contradiction to the quality of bound required by a result such as \cref{thm:effective-distance}. In particular, we obtain graphs where $\mu_{p,G}$ and $\nu_{p,G}$ are both exponentially decaying in the subgraph size $r$ but an estimate of the quality of \cref{thm:effective-distance} does not hold.

\section{Ramsey numbers}\label{sec:ramsey}
Now that we have established effective quasirandomness bounds, our approach to bounding the Ramsey numbers follows a framework developed by \cite{Tho88,Con09}. We compute $K_{r-1}$- and $K_r$-densities effectively, and show a conflict if $R(k+1,\ell+1)/\binom{k+\ell}{k}$ is not ``decreasing'' at some rate.

\subsection{The structure of Ramsey graphs}\label{sub:ramsey-graphs}
Let $\alpha(k,\ell)$ be a symmetric function taking positive values which we will choose later. For any such function, let
\[\alpha^\ast(k,\ell) = \frac{\lfloor\alpha(k,\ell)\binom{k+\ell}{k}\rfloor}{\binom{k+\ell}{k}}.\]
\begin{definition}\label{def:alpha-smooth}
We say a symmetric function $\alpha$ is $(\beta,\gamma)$\emph{-smooth for }$(k,\ell,r)$ if
\begin{align}
\begin{split}\label{eq:smooth}
&R(k+1-m,\ell+1)\le\alpha(k-m,\ell)\binom{k+\ell-m}{\ell},\qquad\frac{\alpha(k-m,\ell)}{\alpha^\ast(k,\ell)}\le 1+m\beta,\\
&R(k+1,\ell+1-m)\le\alpha(k,\ell-m)\binom{k+\ell-m}{k},\qquad\frac{\alpha(k,\ell-m)}{\alpha^\ast(k,\ell)}\le 1+m\gamma
\end{split}
\end{align}
hold for $m\in\{1,2,r-1\}$.
\end{definition}
Our goal is to show that $R(k+1,\ell+1)\le\alpha(k,\ell)\binom{k+\ell}{k}$ by induction. First we note that a graph on $\alpha^\ast(k,\ell)\binom{k+\ell}{k}$ vertices containing no $K_{k+1}$ or $\ol{K}_{\ell+1}$ has degrees and codegrees ``close to random''.
\begin{lemma}[From {\cite[Lemma~3.1]{Con09}}]\label{lem:degree-regularity}
Suppose $\alpha$ is $(\beta,\gamma)$-smooth for $(k,\ell,r)$ and there is a graph $G$ on $n=\alpha^\ast(k,\ell)\binom{k+\ell}{k}$ vertices with no $K_{k+1}$ or $\ol{K}_{\ell+1}$. Let $p = k/(k+\ell)$. Then
\[-(1-p)\gamma\le\mb{E}_yf_{p,G}(x,y)\le p\beta,\qquad\nu_{p,G}\le 2\max(p,1-p)(p\beta+(1-p)\gamma)+n^{-1}.\]
\end{lemma}
\begin{remark}
We also find $\mu_{p,G}\le\max(p\beta,(1-p)\gamma)$. To deduce this from \cite[Lemma~3.1]{Con09}, one must check the different cases of the signs of $\beta,\gamma$; it does in fact follow that, for instance, we do not have $\beta,\gamma < 0$. Also, under these hypotheses, by symmetry of $\alpha$, switching $k,\ell$ shows that $\alpha$ is $(\gamma,\beta)$-smooth for $(\ell,k,r)$. Applying \cref{lem:degree-regularity} to $\ol{G}$ and $\ell/(k+\ell)$ shows that
\[\mu_{1-p,\ol{G}}\le\max(p\beta,(1-p)\gamma),\qquad\nu_{1-p,\ol{G}}\le 2\max(p,1-p)(p\beta+(1-p)\gamma)+n^{-1}.\]
\end{remark}
Using these estimates and our effective quasirandomness estimates from \cref{sub:graph-effective}, we are ready to establish an inductive step. As in \cite{Con09}, we compute statistics for $K_{r-1}$- and $K_r$-densities, and show a violation if $\beta,\gamma$ are not too large and $r$ is small with respect to $k$.
\begin{proposition}\label{prop:inductive-step}
Given $\varepsilon > 0$, there is $c_\varepsilon > 0$ so that the following holds. Suppose $\alpha$ is symmetric and $(\beta,\gamma)$-smooth for $(k,\ell,r)$, and further suppose $\ell/k\in[\varepsilon,1]$. Assume that $|\beta| + |\gamma|\le r(\log k)^2/k$ and $\alpha(k,\ell)\ge\exp(-r(\ell/k)\log k)$. If $r\le c_\varepsilon\log k$ and $k\beta+\ell\gamma\le (r-3)\ell/(2k)$, then
\[R(k+1,\ell+1)\le\alpha(k,\ell)\binom{k+\ell}{k}.\]
\end{proposition}
\begin{proof}
We will ultimately choose $c_\varepsilon$ to be small enough based on various conditions. We will denote by $c_\varepsilon'$ by some positive constant that may depend on $\varepsilon$ and our choice of $c_\varepsilon$; it will potentially change line to line and can be made arbitrarily large by choosing $c_\varepsilon$ sufficiently small.

Suppose for the sake of contradiction that $R(k+1,\ell+1)>\alpha(k,\ell)\binom{k+\ell}{k}$. Then there is a graph $G$ on $n = \alpha^\ast(k,\ell)\binom{k+\ell}{k}$ vertices avoiding $K_{k+1}$ and $\ol{K}_{\ell+1}$. Let $p = k/(k+\ell)$.

By \cref{lem:degree-regularity}, we have
\[\mu_{p,G}\le\max(p\beta,(1-p)\gamma),\qquad\nu_{p,G}\le 2\max(p,1-p)(p\beta+(1-p)\gamma)+n^{-1} < \frac{r}{k+\ell},\]
using $n\ge\lfloor\exp(-r(\ell/k)\log k)\binom{k+\ell}{k}\rfloor > (k+\ell)/2$ (for $c_\varepsilon$ chosen appropriately).

Let $\nu = r/k$, and note that $r\le\log(\nu^{-1})/(12\log(8/\min(p,1-p)))$ for appropriate $c_\varepsilon$ as $p\in[1/2,1/(1+\varepsilon)]$. Let $f = f_{p,G}$. By \cref{thm:effective-distance} we have, letting $\mu = \mu_{p,G}$,
\[p^{-\binom{r-1}{2}}t_{K_{r-1}}(W_G) = 1 + p^{-1}\binom{r-1}{2}t_{K_2}(f) + p^{-3}\binom{r-1}{3}t_{K_3}(f) + O(2^{-2r}\nu^{7/6} + r^4p^{-2}\mu^2).\]
Furthermore, by $|\beta|+|\gamma|\le r(\log k)^2/k$ we have $\mu\le\max(p\beta,(1-p)\gamma)\le\nu(\log k)^2$. Using $\nu(\log k)^4 = O_\varepsilon(2^{-c_\varepsilon'r})$ and $r^{7/6}k^{-1/6} = O_\varepsilon(2^{-c_\varepsilon'r})$ for appropriate $c_\varepsilon'$ (from $r\le c_\varepsilon\log k$ and $\ell\ge\varepsilon k$), we obtain
\begin{equation}\label{eq:Kr-1}
p^{-\binom{r-1}{2}}t_{K_{r-1}}(W_G) = 1 + p^{-1}\binom{r-1}{2}t_{K_2}(f) + p^{-3}\binom{r-1}{3}t_{K_3}(f) + O_\varepsilon(2^{-c_\varepsilon'r}k^{-1})
\end{equation}
and similarly
\begin{equation}\label{eq:Kr}
p^{-\binom{r}{2}}t_{K_r}(W_G) = 1 + p^{-1}\binom{r}{2}t_{K_2}(f) + p^{-3}\binom{r}{3}t_{K_3}(f) + O_\varepsilon(2^{-c_\varepsilon'r}k^{-1})
\end{equation}
for appropriate $c_\varepsilon'$. Also, by taking $c_\varepsilon$ small enough, we can choose $c_\varepsilon'$ as large as we want. Furthermore, since $W_G(x,x) = 0$ we in fact have
\[t_{K_s}(W_G) = n^{-s}\#\{\text{clique }s\text{-tuples in }G\}\]
for each $s\ge 1$. Now note that every $(r-1)$-tuple forming a clique can extend to an $r$-clique in less than $R(k+1-(r-1),\ell+1)$ ways, else by applying the definition of the Ramsey number we find an independent set of size $\ell+1$ or a clique of size $k+1-(r-1)$ which is fully connected to a disjoint $(r-1)$-clique. We deduce
\begin{align*}
t_{K_r}(W_G) &= n^{-r}\#\{\text{clique }r\text{-tuples in }G\}\le n^{-r}R(k+2-r,\ell+1)\#\{\text{clique }(r-1)\text{-tuples in }G\}\\
&=\frac{R(k+2-r,\ell+1)}{n}t_{K_{r-1}}(W_G)\le\frac{\alpha(k-(r-1),\ell)}{\alpha^\ast(k,\ell)}\frac{\binom{k+\ell-(r-1)}{k-(r-1)}}{\binom{k+\ell}{k}}t_{K_{r-1}}(W_G)\\
&\le (1+(r-1)\beta)\prod_{i=0}^{r-2}\frac{k-i}{k+\ell-i}\cdot t_{K_{r-1}}(W_G)\\
&\le (1+(r-1)\beta)p^{r-1}\prod_{i=0}^{r-2}\exp(i/(k+\ell)+i^2/(k+\ell)^2-i/k)\cdot t_{K_{r-1}}(G)\\
&\le (1+(r-1)\beta)p^{r-1}\exp(-(r-1)(r-2)(1-p)/(2k)+r^3/(k+\ell)^2)t_{K_{r-1}}(G),
\end{align*}
using smoothness of $\alpha$ as well as $1/(1-x)\le\exp(x+x^2)$ for $x\in(0,1/2)$. Expanding along with \cref{eq:Kr-1,eq:Kr} and dividing by $p^{\binom{r}{2}}$, we find
\begin{align*}
1 + &p^{-1}\binom{r}{2}t_{K_2}(f) + p^{-3}\binom{r}{3}t_{K_3}(f) + O_\varepsilon(2^{-c_\varepsilon'r}k^{-1})\\
&\le (r-1)\beta - \binom{r-1}{2}\frac{1-p}{k} + 1 + p^{-1}\binom{r-1}{2}t_{K_2}(f) + p^{-3}\binom{r-1}{3}t_{K_3}(f) + O_\varepsilon(2^{-c_\varepsilon'r}k^{-1}),
\end{align*}
noting that expressions such as $(r-1)\beta\cdot O(r^2/k)$ can be absorbed into the error terms by $r\le c_\varepsilon\log k$. Here the $O$'s merely assert that there exist bounded quantities of the claimed form so that the above inequality is true. Subtracting over terms and dividing by $r-1$, we obtain
\begin{equation}\label{eq:beta}
\beta\ge\frac{r-2}{2}\cdot\frac{1-p}{k} + p^{-1}t_{K_2}(f) + p^{-3}\frac{r-2}{2}t_{K_3}(f) + O_\varepsilon(2^{-c_\varepsilon'r}k^{-1}).
\end{equation}

Now, by the remark following \cref{lem:degree-regularity}, we can apply the above arguments to $\ol{G}$ and $1-p$. Letting $\ol{f} = f_{1-p,\ol{G}}$, we obtain
\begin{equation}\label{eq:gamma}
\gamma\ge\frac{r-2}{2}\cdot\frac{p}{\ell}+(1-p)^{-1}t_{K_2}(\ol{f}) + (1-p)^{-3}\frac{r-2}{2}t_{K_3}(\ol{f}) + O_\varepsilon(2^{-c_\varepsilon'r}k^{-1}).
\end{equation}
Note that $f + \ol{f} = W_G + W_{\ol{G}} - 1 = -\mbm{1}_{x=y}$ (recall these functions are defined on the set $V(G)^2$). Thus we find
\[t_{K_2}(\ol{f}) = -t_{K_2}(f) + O(n^{-1}),\qquad t_{K_3}(\ol{f}) = -t_{K_3}(f) + O(n^{-1}).\]
This, combined with the inequality $k^3\cref{eq:beta}+\ell^3\cref{eq:gamma}$, yields
\[k^3\beta + \ell^3\gamma\ge\frac{r-2}{2}k\ell + (k+\ell)(k^2-\ell^2)t_{K_2}(f) + O_\varepsilon(2^{-c_\varepsilon'r}k^2),\]
using that $n$ is large by the lower bound on $\alpha$ (exponential in $k$ so doubly exponential in $r$) to absorb terms into the error term. From the first part of \cref{lem:degree-regularity}, we find $t_{K_2}(f)\ge -(1-p)\gamma$, and recall $k\ge\ell$. Using this and switching terms to the other side, dividing by $k^2$, and absorbing the error term (choosing $c_\varepsilon$ small enough so that $c_\varepsilon'$ is sufficiently large), we obtain the contradiction
\[k\beta+\ell\gamma > \frac{r-3}{2}\cdot\frac{\ell}{k}.\qedhere\]
\end{proof}

\subsection{Induction}\label{sub:induction}
Note that \cref{prop:inductive-step} provides a way of bootstrapping bounds on Ramsey numbers, as long as the (symmetric) function $\alpha$ is smooth with respect to parameters that are not too large. This means that iterating it will give some amount of improvement over the Erd\"{o}s-Szekeres bound \eqref{eq:erdos-szekeres}, and the rest is merely an exercise in extracting the behavior of some recurrence. For our purposes the induction scheme presented in \cite{Con09}, which essentially generalizes the one in \cite{Tho88}, will suffice.

The key point is that a function of the form $\alpha(x,y)=\exp(-\rho(y/x)\log(x+y))$ satisfies
\begin{equation}\label{eq:pde}
-x\frac{\partial}{\partial x}[\log\alpha(x,y)] - y\frac{\partial}{\partial y}[\log\alpha(x,y)] = \rho(y/x),
\end{equation}
which is the continuous analogue of the crucial smoothness condition in \cref{prop:inductive-step} (if say $\rho(x)\le (r-3)x/2$).
\begin{definition}
Let $\tau(x) = 6x^5 - 15x^4 + 10x^3$. For $r\ge 5$ and $\varepsilon\in(0,1/2)$ we define the function $\rho_{r,\varepsilon}\colon [0,+\infty)\to[0,+\infty)$ via
\[\rho_{r,\varepsilon}(x) = \begin{cases}0&\text{if }x\in[0,\varepsilon]\\(r-4)\tau((x-\varepsilon)/(1-\varepsilon))/4&\text{if }x\in[\varepsilon,1]\\\rho_{r,\varepsilon}(1/x)&\text{if }x\in[1,+\infty)\end{cases}\]
and the function $\phi$ via $\phi_{r,\varepsilon}(k,\ell) = \rho_{r,\varepsilon}(\ell/k)\log(k+\ell)$.
\end{definition}
We will use $\alpha(x,y) = C_{r,\varepsilon}\exp(-\phi_{r,\varepsilon}(x,y))$ for some appropriate $C_{r,\varepsilon}$ and value of $r$ to be chosen later. Notice that this choice of $\alpha$ is symmetric. We collect the following lemmas, which provide the necessary bounds for the discrete version of \cref{eq:pde}.
\begin{lemma}[{\cite[Lemma~5.1]{Con09}}]\label{lem:rho-bound}
For $r\ge 5$ and $\varepsilon\in(0,1/2)$, $\rho_{r,\varepsilon}$ is twice-differentiable and satisfies $\rho_{r,\varepsilon}(x)\in[0,(r-4)x/2]$ for $x\in[0,1]$ and $\snorm{\rho'}_\infty\le r$, $\snorm{\rho''}_\infty\le 10r$.
\end{lemma}
\begin{lemma}[{\cite[Lemma~5.2]{Con09}}]\label{lem:alpha-smooth}
If $k,\ell\ge 200r^4/\varepsilon^2$ and if
\[b = \frac{4\rho_{r,\varepsilon}(\ell/k)+\varepsilon}{4(k+\ell)} - \frac{\ell\log(k+\ell)}{k^2}\rho_{r,\varepsilon}'(\ell/k),\qquad c = \frac{4\rho_{r,\varepsilon}(\ell/k)+\varepsilon}{4(k+\ell)} + \frac{\log(k+\ell)}{k}\rho_{r,\varepsilon}'(\ell/k),\]
then
\[\exp(\phi_{r,\varepsilon}(k,\ell)-\phi_{r,\varepsilon}(k-m,\ell))\le 1+mb,\qquad\exp(\phi_{r,\varepsilon}(k,\ell)-\phi_{r,\varepsilon}(k,\ell-m))\le 1+mc\]
for $m\in\{1,2,r-1\}$.
\end{lemma}
Now we are ready to establish a bound for Ramsey numbers depending on the parameter $r$ as well as $\varepsilon$ (which controls what regime of $\ell/k$ this bound is nontrivial for).
\begin{theorem}\label{thm:ramsey-r}
Let $r\ge 5$ and $\varepsilon\in(0,1/2)$. Then there is $C_\varepsilon > 0$ with
\begin{equation}\label{eq:ramsey-r}
R(k+1,\ell+1)\le 2^{C_\varepsilon r^2}\exp(-\phi_{r,\varepsilon}(k,\ell))\binom{k+\ell}{k}.
\end{equation}
\end{theorem}
\begin{proof}
Let $\alpha(x,y) = 2^{C_\varepsilon r^2}\exp(-\phi_{r,\varepsilon}(k,\ell))$, where $C_\varepsilon > 0$ will be chosen later.

First we verify \cref{eq:ramsey-r} if $\min(k,\ell)\le 2^{C_\varepsilon r}$. Noting that $\phi_{r,\varepsilon}$ hence $\alpha$ is symmetric, as are the Ramsey numbers, it suffices to check it for $k\ge\ell$ and $\ell\le 2^{C_\varepsilon r}$. By \cref{lem:rho-bound} we have
\[\alpha(k,\ell) = 2^{C_\varepsilon r^2}\exp(-\rho_{r,\varepsilon}(\ell/k)\log(k+\ell))\ge 2^{C_\varepsilon r^2}(k+\ell)^{-(r-4)\ell/(2k)}\ge 2^{C_\varepsilon r^2}(2\ell)^{-(r-4)/2}\ge 1,\]
where the second inequality uses that $(k+\ell)^{\ell/k}$ is decreasing in $k$ hence achieves its maximum in the region $k\ge\ell$ (fixing $\ell$) when $k = \ell$. The result follows from the Erd\H{o}--Szekeres bound \cref{eq:erdos-szekeres}.

Next we verify \cref{eq:ramsey-r} if $\min(k/\ell,\ell/k) < \varepsilon$. By symmetry we can assume $\ell\le k$, hence $\ell/k\le\varepsilon$. In this case, $\rho_{r,\varepsilon}(\ell/k) = 0$, so $\alpha(k,\ell) = 2^{C_\varepsilon r}\ge 1$ and again the result follows from the Erd\H{o}--Szekeres bound \cref{eq:erdos-szekeres}.

Now suppose $\min(k,\ell)\ge 2^{C_\varepsilon r}$ and $\min(k/\ell,\ell/k)\ge\varepsilon$. We claim that the bound \cref{eq:ramsey-r} follows from \cref{eq:ramsey-r} for $(k-m,\ell)$ and $(k,\ell-m)$ when $m\in\{1,2,r-1\}$. This clearly finishes, since iterating this yields the result (noting that the above took care of any necessary base cases for such an iteration). Without loss of generality we can suppose $\ell\le k$ by symmetry, and thus in fact $\ell/k\in[\varepsilon,1]$. Let $n = \lfloor\alpha(k,\ell)\binom{k+\ell}{k}\rfloor$.

In order to prove this claim, it suffices to check that $\alpha$ satisfies the conditions of \cref{prop:inductive-step}, namely, for some $\beta,\gamma\in\mb{R}$ we need that $\alpha$ is $(\beta,\gamma)$-smooth for $(k,\ell,r)$, that $|\beta|+|\gamma|\le r(\log k)^2/k$, that $\alpha(k,\ell)\ge\exp(-r(\ell/k)\log k)$, that $r\le c_{\varepsilon,\ref{prop:inductive-step}}\log k$, and that $k\beta + \ell\gamma\le (r-3)\ell/(2k)$.

We have $\phi_{r,\varepsilon}(\ell/k)\le (r-4)/2\cdot(\ell/k)\log(k+\ell)$ by \cref{lem:rho-bound}, so the condition $\alpha(k,\ell)\ge\exp(-r(\ell/k)\log k)$ is satisfied (since, e.g., $k+\ell\le k^2$). If $C_\varepsilon$ is chosen sufficiently large, we see $r\le c_{\varepsilon,\ref{prop:inductive-step}}\log k$ will hold, and also $\min(k,\ell)\ge 2^{C_\varepsilon r}\ge 200r^4/\varepsilon^2$. Thus, by \cref{lem:alpha-smooth}, we have for
\begin{align*}
b = \frac{4\rho_{r,\varepsilon}(\ell/k)+\varepsilon}{4(k+\ell)} - \frac{\ell\log(k+\ell)}{k^2}\rho_{r,\varepsilon}'(\ell/k),&\qquad c = \frac{4\rho_{r,\varepsilon}(\ell/k)+\varepsilon}{4(k+\ell)} + \frac{\log(k+\ell)}{k}\rho_{r,\varepsilon}'(\ell/k),\\
\beta = \frac{2\rho_{r,\varepsilon}(\ell/k)+\varepsilon}{2(k+\ell)} - \frac{\ell\log(k+\ell)}{k^2}\rho_{r,\varepsilon}'(\ell/k),&\qquad\gamma = \frac{2\rho_{r,\varepsilon}(\ell/k)+\varepsilon}{2(k+\ell)} + \frac{\log(k+\ell)}{k}\rho_{r,\varepsilon}'(\ell/k)
\end{align*}
and for all $m\in\{1,2,r-1\}$ that
\[\frac{\alpha(k-m,\ell)}{\alpha^\ast(k,\ell)}\le(1+1/n)\frac{\alpha(k-m,\ell)}{\alpha(k,\ell)}\le (1+1/n)(1+mb)\le 1+m\beta\]
and similarly
\[\frac{\alpha(k,\ell-m)}{\alpha^\ast(k,\ell)}\le 1+m\gamma.\]
Here we used that $n$ is significantly larger than $k+\ell$, valid given the bounds on $k,\ell,\alpha$. Therefore, for this choice of $\beta,\gamma$, we see that $\alpha$ is $(\beta,\gamma)$-smooth for $(k,\ell,r)$. Furthermore, we see that $|\beta|$ and $|\gamma|$ are bounded in magnitude by $O_\varepsilon(r(\log k)/k)$ due to their definitions and \cref{lem:rho-bound}, hence for $k$ sufficiently large (i.e., $C_\varepsilon$ sufficiently large) we have $|\beta|+|\gamma|\le r(\log k)^2/k$.

Therefore it suffices to verify that $k\beta + \ell\gamma\le (r-3)\ell/(2k)$, and the proof will be completed. But using the explicit values above, we compute
\[k\beta + \ell\gamma = \rho_{r,\varepsilon}(\ell/k) + \frac{\varepsilon}{2}\le\frac{r-4}{2}\cdot\frac{\ell}{k}+\frac{\varepsilon}{2}\le\frac{r-3}{2}\cdot\frac{\ell}{k},\]
using \cref{lem:rho-bound} and $\ell/k\in[\varepsilon,1]$. We are finished.
\end{proof}
Finally, we prove \cref{thm:main-2}. We note that
\[\rho_{r,\varepsilon/2}(x)\ge\frac{(r-4)\varepsilon^3}{32}\]
for $\varepsilon\in(0,1/2)$ and $x\in[\varepsilon,1]$.
\begin{proof}[Proof of \cref{thm:main-2}]
Suppose $\ell/k\in[\varepsilon,1]$. By \cref{thm:ramsey-r}, we have for any $r\ge 5$ that, letting $C_\varepsilon = C_{\varepsilon/2,\ref{thm:ramsey-r}}$,
\[R(k+1,\ell+1)\le 2^{C_\varepsilon r^2}\exp(-\phi_{r,\varepsilon/2}(k,\ell))\binom{k+\ell}{k}\le 2^{C_\varepsilon r^2}\exp(-(r-4)\varepsilon^3\log k/32)\binom{k+\ell}{k}.\]
Now, choosing $r = \varepsilon^3\log k/(64C_\varepsilon)$ (assuming $k$ is large enough that $r\ge 5$), we obtain
\[R(k+1,\ell+1)\le e^{-c_\varepsilon(\log k)^2}\binom{k+\ell}{k}\]
for appropriate $c_\varepsilon > 0$.
\end{proof}

\bibliographystyle{amsplain0.bst}
\bibliography{main.bib}

\providecommand{\bysame}{\leavevmode\hbox to3em{\hrulefill}\thinspace}
\providecommand{\MR}{\relax\ifhmode\unskip\space\fi MR }
\providecommand{\MRhref}[2]{%
  \href{http://www.ams.org/mathscinet-getitem?mr=#1}{#2}
}
\providecommand{\href}[2]{#2}
\begin{thebibliography}{10}

\bibitem{BMS15}
J.~Balogh, R.~Morris, and W.~Samotij, \emph{Independent sets in hypergraphs},
  J. Amer. Math. Soc. \textbf{28} (2015), 669--709.

\bibitem{CGW89}
F.~R.~K. Chung, R.~L. Graham, and R.~M. Wilson, \emph{Quasi-random graphs},
  Combinatorica \textbf{9} (1989), 345--362.

\bibitem{CGSS14}
D.~Conlon, W.~T. Gowers, W.~Samotij, and M.~Schacht, \emph{On the {K\L R}
  conjecture in random graphs}, Israel J. Math. \textbf{203} (2014), 535--580.

\bibitem{Con09}
D.~Conlon, \emph{A new upper bound for diagonal {R}amsey numbers}, Ann. of
  Math. (2) \textbf{170} (2009), 941--960.

\bibitem{CFS15}
D.~Conlon, J.~Fox, and B.~Sudakov, \emph{Recent developments in graph {R}amsey
  theory}, Surveys in combinatorics 2015, London Math. Soc. Lecture Note Ser.,
  vol. 424, Cambridge Univ. Press, Cambridge, 2015, pp.~49--118.

\bibitem{CFZ14}
D.~Conlon, J.~Fox, and Y.~Zhao, \emph{Extremal results in sparse pseudorandom
  graphs}, Adv. Math. \textbf{256} (2014), 206--290.

\bibitem{CFZ15}
D.~Conlon, J.~Fox, and Y.~Zhao, \emph{A relative {S}zemer\'{e}di theorem},
  Geom. Funct. Anal. \textbf{25} (2015), 733--762.

\bibitem{E47}
P.~Erd\"{o}s, \emph{Some remarks on the theory of graphs}, Bull. Amer. Math.
  Soc. \textbf{53} (1947), 292--294.

\bibitem{ES35}
P.~Erd\"{o}s and G.~Szekeres, \emph{A combinatorial problem in geometry},
  Compositio Math. \textbf{2} (1935), 463--470.

\bibitem{GS05}
S.~Gerke and A.~Steger, \emph{The sparse regularity lemma and its
  applications}, Surveys in combinatorics 2005, London Math. Soc. Lecture Note
  Ser., vol. 327, Cambridge Univ. Press, Cambridge, 2005, pp.~227--258.

\bibitem{Gow97}
W.~T. Gowers, \emph{Lower bounds of tower type for {S}zemer\'{e}di's uniformity
  lemma}, Geom. Funct. Anal. \textbf{7} (1997), 322--337.

\bibitem{Gow07}
W.~T. Gowers, \emph{Hypergraph regularity and the multidimensional
  {S}zemer\'{e}di theorem}, Ann. of Math. (2) \textbf{166} (2007), 897--946.

\bibitem{GR87}
R.~L. Graham and V.~R\"{o}dl, \emph{Numbers in {R}amsey theory}, Surveys in
  combinatorics 1987 ({N}ew {C}ross, 1987), London Math. Soc. Lecture Note
  Ser., vol. 123, Cambridge Univ. Press, Cambridge, 1987, pp.~111--153.

\bibitem{Lov11}
L.~Lov\'{a}sz, \emph{Subgraph densities in signed graphons and the local
  {S}imonovits-{S}idorenko conjecture}, Electron. J. Combin. \textbf{18}
  (2011), Paper 127, 21.

\bibitem{Rad94}
S.~P. Radziszowski, \emph{Small {R}amsey numbers}, Electron. J. Combin.
  \textbf{1} (1994), Dynamic Survey 1, 30.

\bibitem{Ram29}
F.~P. Ramsey, \emph{On a {P}roblem of {F}ormal {L}ogic}, Proc. London Math.
  Soc. (2) \textbf{30} (1929), 264--286.

\bibitem{RNSSK05}
V.~R\"{o}dl, B.~Nagle, J.~Skokan, M.~Schacht, and Y.~Kohayakawa, \emph{The
  hypergraph regularity method and its applications}, Proc. Natl. Acad. Sci.
  USA \textbf{102} (2005), 8109--8113.

\bibitem{SSSZ20}
A.~Sah, M.~Sawhney, D.~Stoner, and Y.~Zhao, \emph{A reverse {S}idorenko
  inequality}, Inventiones mathematicae (2020), online.

\bibitem{Sch16}
M.~Schacht, \emph{Extremal results for random discrete structures}, Ann. of
  Math. (2) \textbf{184} (2016), 333--365.

\bibitem{Spe77}
J.~Spencer, \emph{Asymptotic lower bounds for {R}amsey functions}, Discrete
  Math. \textbf{20} (1977/78), 69--76.

\bibitem{Sze78}
E.~Szemer\'{e}di, \emph{Regular partitions of graphs}, Probl\`emes
  combinatoires et th\'{e}orie des graphes ({C}olloq. {I}nternat. {CNRS},
  {U}niv. {O}rsay, {O}rsay, 1976), Colloq. Internat. CNRS, vol. 260, CNRS,
  Paris, 1978, pp.~399--401.

\bibitem{Tho87}
A.~Thomason, \emph{Pseudorandom graphs}, Random graphs '85 ({P}ozna\'{n},
  1985), North-Holland Math. Stud., vol. 144, North-Holland, Amsterdam, 1987,
  pp.~307--331.

\bibitem{Tho88}
A.~Thomason, \emph{An upper bound for some {R}amsey numbers}, J. Graph Theory
  \textbf{12} (1988), 509--517.

\bibitem{Zha17}
Y.~Zhao, \emph{Extremal regular graphs: independent sets and graph
  homomorphisms}, Amer. Math. Monthly \textbf{124} (2017), 827--843.

\end{thebibliography}

\end{document}